\documentclass[11pt]{amsart}
\usepackage{amsmath,amssymb,amsbsy,amsfonts,latexsym,amsopn,amstext,cite,
	amsxtra,euscript,amscd,bm}
\usepackage{amsmath,amssymb,amsbsy,amsfonts,amsthm,latexsym,
	amsopn,amstext,amsxtra,euscript,amscd,mathrsfs}
\usepackage{mathtools}
\usepackage{todonotes}
\usepackage{url}
\usepackage[colorlinks,linkcolor=blue,anchorcolor=blue]{hyperref}

\makeindex \makeatletter
\def\captionof#1#2{{\def\@captype{#1}#2}}
\makeatother

\newcounter{tablegroup}
\newcounter{subtable}[tablegroup]

\newcommand{\tend}[3][]{\xrightarrow[#2\to#3]{#1}}

\newtheorem{thm}{Theorem}[section]
\newtheorem{cor}[thm]{Corollary}
\newtheorem{lem}[thm]{Lemma}
\newtheorem{prop}[thm]{Proposition}
\newtheorem{defn}[thm]{Definition}
\newtheorem{Que}[thm]{Question}

\numberwithin{equation}{section}

\newcommand{\eps}{\varepsilon}

\newcommand{\bmu}{\bm \mu}

\newcommand{\bml}{\bm \lambda}

\begin{document}
\title[M\"{o}bius disjointness conjecture for local dendrite maps]
{M\"{o}bius disjointness conjecture for local dendrite maps}

\author[el.H. el Abdalaoui, G. Askri,
H. Marzougui]{el Houcein el Abdalaoui, Ghassen Askri,
Habib Marzougui}

\address{e. H. el Abdalaoui, Normandy University of Rouen,
Department of Mathematics, LMRS  UMR 6085 CNRS, Avenue de l'Universit\'e, BP.12,
76801 Saint Etienne du Rouvray - France.}
\email{elhoucein.elabdalaoui@univ-rouen.fr}
\address{ Ghassen Askri, University of Carthage, Bizerte Preparatory Engineering Institute, 7021, Jarzouna, Tunisia}
\email{askri.ghassen@gmail.com}
\address{ Habib Marzougui, University of Carthage, Faculty
of Sciences of Bizerte, Department of Mathematics,
Jarzouna, 7021, Tunisia;}
\email{hmarzoug@ictp.it and habib.marzougui@fsb.rnu.tn}



\begin{abstract}
We prove that the M\"{o}bius disjointness conjecture holds for graph maps and for all
monotone local dendrite maps. We further show that this also hold for continuous map on certain class of dendrites. Moreover, we see that there is a example of transitive dendrite map with zero entropy  for which M\"{o}bius disjointness holds. 
\end{abstract}
\maketitle

\section{\bf Introduction}
Let $X$ be a compact metric space with a metric $d$ and let $f: X\to X$ be a continuous map.
We call for short ($X, f$) a dynamical system. The topological entropy $h(f)$ of such a system is defined
as: \[h(f) = \underset{\varepsilon\to 0}\lim \underset{n\to +\infty}\limsup \dfrac{1}{n}
\textrm{log}~\textrm{sep}(n, f, \varepsilon).\]
where for $n$ integer and $\varepsilon>0$, $\textrm{sep}(n, f, \varepsilon)$ is the maximal possible cardinality of
an ($n, f, \varepsilon$)-separated set in $X$, this later means that for every
two points of it, there exists $0\leq j<n$ with $d(f^j(x), f^j(y)) > \varepsilon$, where $f^{j}$ denotes
the $j$-$\textrm{th}$ iterate of $f$. A dynamical system ($X, f$) is called a \textit{null system} if its sequence entropy is zero
for any sequence; we refer the reader to \cite{G}, \cite{HMY} for the details. 
The M\"{o}bius function $\bmu$ is an ally of the Liouville function $\bml$.
This later function is defined by $\bml(n) =1$ if the number of prime factors of $n$ is even and $-1$ otherwise.
Precisely, the M\"{o}bius function is given by
$$\bmu(n)  =  \begin{cases}
 1  & \ \textrm{if } n = 1 \\
 \bml(n) & \ \textrm{if } \textrm{all primes in decomposition of $n$ are distincts} \, \\
 0 & \ \textrm{otherwise}.
 \end{cases}$$

In 2010, P. Sarnak \cite{Sa}, \cite{Sa2} initiated the study of the dynamical system generated by the M\"{o}bius function, and in the connection with the M\"{o}bius randomness law, he stated the following conjecture:

\textbf{ Sarnak's Conjecture.} Let ($X, f$) be a dynamical system with zero topological entropy $h(f)=0$. Then

\begin{equation}\label{(1.1)}
    S_N(x,\varphi): = \frac{1}{N} \sum_{n=1}^N \bmu(n)\varphi(f^n (x)) = o(1), \textrm{as } \ N\to +\infty
\end{equation}
 for each $x\in X$ and each continuous function $\varphi : X\longrightarrow \mathbb{R}$.\\
 
 We recall that the M\"{o}bius randomness law \cite{Kowalski} assert that for any ``reasonable'' sequence $(a_n)$, we have
 $$\frac{1}{N}\sum_{n=1}^{N}\bmu(n) a_n=o(1).$$
 
It turns out that Sarnak's conjecture \eqref{(1.1)} is connected to the popular Chowla conjecture on the multiple autocorrelations of the M\"{o}bius function. This later conjecture assert that for any $r\geq 0$, $1\leq a_1<\dots<a_r$, $i_s\in \{1,2\}$ not all equal to $2$, we have
	\begin{equation}\label{cza}
	\sum_{n\leq N}\bmu^{i_0}(n)\bmu^{i_1}(n+a_1)\cdot\ldots\cdot\bmu^{i_r}(n+a_r)={\rm o}(N).
	\end{equation} 
The Chowla conjecture implies a weaker conjecture stated by Chowla in \cite{chowla}. We refer to \cite{chowla} for the statement of this weaker form of Chowla conjecture. For more details on the connection between Sarnak and Chowla conjectures we refer to the
very recent works of the first author \cite{elabdal-SC}, Tao \cite{blogTaoII}, Gomilko-Kwietniak-Lema\'{n}czyk \cite{lem} and Tao and Ter\"{a}v\"{a}inen~\cite{TaoTer}.\\

\noindent Note, that in the simplest case, when $f \equiv \textrm{const}$, ($1.1)$ is equivalent to the statement
\begin{equation*}\label{prime}
 \frac{1}{N} \sum_{n=1}^N \bmu(n) = o(1),  \textrm{as } \ N\to +\infty
\end{equation*}
which is equivalent to the Prime Number Theorem \cite{Apostol}. The conjecture ($1.1)$, also known as the M\"{o}bius disjointness conjecture
is known to be true for several dynamical systems, see e.g.
[\cite{ALR},\cite{AKLR},\cite{HWZ}, \cite{HMY}, \cite{Ka}, \cite{LS},\cite{Mullner},\cite{FKL}] and the references therein. In \cite{Ka}, Karagulyan proved the conjecture for the
orientation preserving circle-homeomorphisms and for continuous
interval maps of zero entropy. 
In the present paper, we are interested in another natural classes of dynamical systems: the graph, dendrite 
and local dendrite maps. We thus establish that for the graph maps and for 
all monotone local dendrite maps Sarnak's conjecture holds. 
We are also able to prove that the M\"{o}bius disjointness property holds for a deterministic class of dendrites for which the set of endpoints is closed and its derived set is finite. 
This extends Karagulyan result on the M\"{o}bius disjointness of any interval maps with zero entropy and (orientation preserving) 
circle homeomorphisms.\\
Recent interests in dynamics on graphs and local dendrites is motivated by the fact that graphs and local dendrites are 
examples of Peano continua with complex topology structures (e.g., \cite{Nadler},
pp. 165-187). On the other hand, dendrites often appear as Julia sets in complex dynamics (see
\cite{b}). After finishing this version, we learned that Li, Oprocha, Yang, and Zeng had recently
solved the conjecture for graph maps \cite{LOYZ}. Notice that our proof of Theorem \ref{t31} and that of \cite{LOYZ} are different. 
Indeed, in their proof, they need a more stronger dynamical property based on the notion of locally mean equicontinuous.\\

In this paper, we also investigate the M\"{o}bius disjointness of transitive dendrite maps. It is turns out that we are able to establish the M\"{o}bius disjointness of the transitive dendrite maps with zero topological entropy introduced by J. Byszewski and \textit{al.} \cite{BFK}.\\

According to the recent result of J. Li, P. Oprocha and G. Zhang \cite{JOZ}, our investigation can be seen as a deep investigation on Sarnak's conjecture. Indeed, the authors therein proved that if the M\"{o}bius disjointness holds for any Gehman dendrite map with zero entropy then Sarnak's conjecture holds.\\
   
We further discuss the problem of M\"{o}bius distinctness for the transitive dendrite map with positive topological entropy introduced by \v{S}pitalsky \cite{Sp}. 
At this point, let us point out that Sarnak mentioned in his paper \cite{Sa} 
that J. Bourgain constructed a topological dynamical system with positive topological entropy for 
which the M\"{o}bius disjointness holds. Later, Downarowicz and Serafin constructed a class of topological systems with the positive topological entropy  which satisfy the M\"{o}bius randomness law \cite{Dow}. 
\medskip

The plan of the paper is as follows. In Section 2, we give some definitions and preliminary properties
on graphs, dendrites and local dendrites which are useful for the rest of the paper.
 Section 3 is devoted to the proof of Theorem \ref{t31} for graph maps of zero entropy. Section 4 is
 devoted to local dendrite maps of zero entropy. In Subsection 4.1 we will prove Theorem
 \ref{t41} for monotone local dendrite maps. In Subsection 4.2, we prove the conjecture for continuous map 
 on a certain class of dendrites. Subsection 4.3, is devoted to the conjecture for an example of transitive 
 dendrite map of zero entropy. 
 Finally, in Subsection 4.4, we discuss the  M\"{o}bius disjointness of an example of transitive dendrite map with positive entropy.
\medskip

\section{\bf Preliminaries and some results}
Let $\mathbb{Z},\ \mathbb{Z}_{+}$ and $\mathbb{N}$ be the sets of integers, non-negative integers and positive integers,
respectively. For $n\in \mathbb{Z_{+}}$ denote by  $f^{n}$ the $n$-$\textrm{th}$ iterate of $f$; that is,
$f^{0}$=identity and $f^{n} = f\circ f^{n-1}$ if $n\in \mathbb{N}$. For any $x\in X$, the subset
$\textrm{Orb}_{f}(x) = \{f^{n}(x): n\in\mathbb{Z}_{+}\}$ is called the \textit{orbit} of $x$ (under $f$).
A subset $A\subset X$ is called \textit{$f-$invariant} (resp. strongly $f-$invariant)
if $f(A)\subset A$ (resp., $f(A) = A$). It is called \textit{a minimal set of $f$} if it is
non-empty, closed, $f$-invariant and minimal (in the sense of
inclusion) for these properties, this is equivalent to say that it is an orbit closure that contains no smaller one;
for example a single finite
orbit. When $X$ itself is a minimal set, then we say that $f$ is \textit{minimal}.
We define the \textit{$\omega$-limit} set of a point $x$ to be the set:
$\omega_{f}(x)  = \{y\in X: \exists\ n_{i}\in \mathbb{N},
n_{i}\rightarrow\infty, \underset{i\to +\infty}\lim d(f^{n_{i}}(x), y) = 0\}$. A point $x\in X$ is called : \\
$-$ \textit{periodic} of period $n\in\mathbb{N}$ if
$f^{n}(x)=x$ and $f^{i}(x)\neq x$ for $1\leq i\leq n-1$; if $n = 1$, $x$ is called a \textit{fixed point} of $f$ i.e.
$f(x) = x$. \\
$-$ \textit{Almost periodic} if for any neighborhood $U$ of $x$ there is $N\in\mathbb{N}$ such that
$\{f^{i+k}(x): i=0,1,\dots,N\}\cap U\neq \emptyset$, for all $k\in \mathbb{N}$. 
It is well known (see e.g. \cite{blo}, Chapter V, Proposition 5) that a point
$x$ in $X$ is almost periodic if and only if $\overline{\textrm{Orb}_{f}(x)}$ is a minimal set of $f$.\\

A pair $(x,y) \in X \times X$ is called
proximal if  $\liminf_{n \rightarrow +\infty}d(f^n(x),f^n(y)) = 0$;  it is called asymptotic if $\lim_{n \rightarrow +\infty }d(f^n(x),f^n(y)) = 0$.  A pair $(x,y) \in X \times X$ is
is said to be a Li-Yorke pair of $f$ if it is proximal but not asymptotic.\\

In this section, we recall some basic properties of graphs dendrites and local dendrites.

A \emph{continuum} is a compact connected metric space. An \emph{arc} is
any space homeomorphic to the compact interval $[0,1]$. A topological space
is \emph{arcwise connected} if any two of its points can be joined by an
arc. We use the terminologies from Nadler \cite{Nadler}.
\medskip

By a \textit{graph} $X$, we mean a continuum which can be written as the union of finitely many arcs
such that any two of them are either
disjoint or intersect only in one or both of their endpoints.
 For any point $v$ of $X$, the \textit{order} of $v$, 
denoted by $\textrm{ord}(v)$, is an integer $r\geq 1$ such that $v$ admits a neighborhood $U$ in $X$
homeomorphic to the set $\{z\in \mathbb{C}: z^{r}\in [0,1]\}$ with
the natural topology, with the
homeomorphism mapping $v$ to $0$. If $r\geq 3$ then $v$ is called a \textit{branch point}. If $r=1$, then we call $v$ an
\textit{endpoint} of $X$. If $r=2$, $v$ is called 
\textit{a regular point} of $X$. 

 Denote by $B(X)$ and $E(X)$ the sets of
branch points and endpoints of $X$ respectively. An edge is the
closure of some connected component of $X\setminus B(X)$, it is
homeomorphic to $[0,1]$. A subgraph of $X$ is a subset of $X$ which
is a graph itself. Every sub-continuum of a graph is a graph
(\cite{Nadler}, Corollary 9.10.1).
Denote by $S^{1}=[0,1]_{\mid 0\sim 1}$ the unit circle endowed with the
orientation: the counter clockwise sense induced via the natural
projection $[0,1]\rightarrow S^{1}$. A circle is any space homeomorphic to $S^{1}$.
\medskip

By a \textit{dendrite} $X$, we mean a locally connected continuum
containing no homeomorphic copy to a circle. Every sub-continuum of a
dendrite is a dendrite (\cite{Nadler}, Theorem 10.10) and every
connected subset of $X$ is arcwise connected (\cite{Nadler},
Proposition 10.9). In addition, any two distinct points $x,y$ of a
dendrite $X$ can be joined by a unique arc with endpoints $x$ and
$y$, denote this arc by $[x,y]$ and let denote by
$[x,y) = [x,y]\setminus\{y\}$ (resp. $(x,y] = [x,y]\setminus\{x\}$ and
$(x,y) = [x,y]\setminus\{x,y\}$). A point $x\in X$ is called an
\textit{endpoint} if $X\setminus\{x\}$ is connected. It is called a
\textit{branch point} if $X\setminus \{x\}$ has more than two
connected components. The number of connected components of $X\setminus \{x\}$ is called the \textit{order} of $x$ and 
denoted by ord$(x)$. The order of $x$ relatively to a subdendrite $Y$ of $X$ is denoted by $ord_Y(x)$.
Denote by $E(X)$ and $B(X)$ the sets of
endpoints, and branch points of $X$, respectively.
By (\cite{Kur}, Theorem 6, 304 and Theorem 7, 302), $B(X)$ is at most countable. A point $x\in
X\setminus E(X)$ is called a \textit{cut point}. It is known that the  set of cut
points of $X$ is dense in $X$ (\cite{Kur}, VI, Theorem 8, p. 302).
Following (\cite{Ar}, Corollary 3.6), for any dendrite $X$, we have
B($X)$ is discrete whenever E($X)$ is closed.  An arc $I$ of $X$ is called \emph{free} if $I \cap B(X)=\emptyset$.
For a subset $A$ of $X$, we call \emph{the convex hull} of $A$, denoted by $[A]$, the intersection of all
sub-continua of $X$ containing $A$, one can write $[A] = \cup_{x, y\in A}[x,y]$.
\medskip

By a \textit{local dendrite} $X$, we mean a continuum every point of which has a dendrite neighborhood. 
A local dendrite is then a locally connected continuum containing only a finite number of
circles (\cite{Kur}, Theorem 4, p. 303). As a consequence every sub-continuum of a local
dendrite is a local dendrite (\cite{Kur}, Theorems 1 and 4, p. 303). Every graph and every dendrite is a local
dendrite.
A continuous map from a local dendrite (resp. graph, resp. dendrite) 
into itself is called a\textit{ local dendrite map} (resp. \textit{graph map}, resp. \textit{dendrite map}).
It is well known that every dendrite map has a fixed point
(see \cite{Nadler}). If $A$ is a sub-dendrite of $X$, define the retraction (or the \textit{first point map}) $r_{A} : X \rightarrow A$
by letting $r_{A}(x) = x$, if $x\in A$, and by letting $r_{A}(x)$ to be the unique point $r_{A}(x)\in A$ such that
$r_{A}(x)$ is a point of any arc in $X$
from $x$ to any point of $A$, if $x\notin A$ (see \cite{Nadler}, p. 176).
Note that the map $r_{A}$ is constant on each
connected component of $X\backslash A$.
For a subset $A$ of $X$, denote by $\overline{A}$ the closure of $A$ and by $\textrm{diam}(A)$ the diameter of
$A$. \\
For every topological space $X$, a map $f: ~X\to X$ is called \textit{monotone} if $f^{-1}(C)$ is connected for any connected
subset $C$ of $X$. In particular, if $f$ is a homeomorphism then it is monotone. Notice that when $X$ is a dendrite, the map $r_{A}$ (above) is monotone.
\medskip
We need the following lemmas.
\begin{lem}[\cite{MS}, Lemma 2.3]\label{l2} Let $X$ be a dendrite,
	$(C_{i})_{i\in\mathbb{N}}$ be a sequence of pairwise disjoint connected subsets of $X$.
	Then $\underset{n\to +\infty}\lim \mathrm{diam}(C_{n})=0$.
\end{lem}


\begin{lem}[\cite{MS}, Lemma 2.1]\label{diam} Let $X$ be a dendrite with metric $d$. Then for any $\eps>0$ there is $0<\delta<\eps$ such that if $d(x,y)<\delta$ then
	$\mathrm{diam}([x,y])<\eps$.
\end{lem}
\medskip

\noindent Theorem 3.3 from \cite{Ar} allows us to deduce the following Lemma:
\medskip

\begin{lem}\label{l12} If $X$ is a dendrite such that $E(X)$ is closed then the order of every branch point is finite.
\end{lem}
\medskip
Applying Dirichlet's theorem on primes in arithmetical progressions (see \cite[p.146]{Apostol}), it easy to see that 
if ($x_n )_{n> 0}$ is an eventually periodic sequence (i.e. $x_n = x_{n+m}$ for some 
fixed number $m\in \mathbb{N}$) and for any $n\geq n_0$), then 
	\begin{eqnarray}\label{l11}
		\frac{1}{N} \sum_{n=1}^N \bmu(n)x_n = o(1).
	\end{eqnarray}
We also need the following lemma from \cite{Ka}.
\begin{lem}\label{l13} Let ($x_n )_{n> 0}$ be a sequence of real numbers such that $|x_n| \leq 1$ for any 
$n \geq 1$. Assume that there is $n_0 , k>0$ such that for any $n,m \geq n_0$, if  $x_n \neq x_m$, then 
$|n-m| \geq k$. Then we have $$\limsup_{N\to +\infty} \Big|\frac{1}{N} \sum_{n=1}^N x_n \Big| \leq \frac{1}{k} .$$
\end{lem}
We shall use the following useful property of $\omega$-limit set. 
\begin{lem}[\cite{Aus}, Theorem 3, p. 67] \label{AU}  Let ($X, f$) be a dynamical system. 
Then for each $x\in X$, there exists an almost periodic point $y\in \omega_f(x)$ such that $(x,y)$ is a proximal pair.
\end{lem}
For the asymptotic pair we have
\begin{lem}\label{l12} Let ($X, f$) be a dynamical system and let $x, y  \in X$. If  $S_N(x,\varphi) =
	o(1)$ and $(x, y)$ is asymptotic then $S_N(y,\varphi) = o(1)$.
\end{lem}

\begin{proof}
	Let $\varphi:~X \to \mathbb{R}$ be a continuous map. Fix $\varepsilon >0$. The map $\varphi$ is uniformly continuous on $X$,
	then there is $\alpha>0$ such that for any $u, \ v\in X$ with $d(u,v)<\alpha$, 
	$|\varphi(u)-\varphi(v)|<\frac{\varepsilon}{2}$. 
	Since $\underset{n \to +\infty} \lim d(f^n(x),f^n(y)) = 0$, there is $n_0 >0$ such that for 
	$n\geq n_0$, $d(f^n(x),f^n(y))<\alpha$. So $|\varphi(f^n(x))-\varphi(f^n (y))|<\frac{\varepsilon}{2}$ 
	for any $n \geq n_0$. 
	Let $n_1 \geq n_0$ be such that for any 
	$N > n_1$, $$\frac{1}{N}\sum_{n=1}^{n_0 -1} |\varphi(f^n(x))-\varphi(f^n(y))|<\frac{\varepsilon}{2}.$$ 
	Then for any $N>n_1$, we have
	\begin{align*}
	|S_N (x,\varphi)-S_N (y,\varphi)| & = 
	\Big|\frac{1}{N}\sum_{n=1}^N \bmu(n)(\varphi(f^n(x))-\varphi(f^n(y)))\Big|\\
	& \leq \frac{1}{N}\sum_{n=1}^{n_0 -1}|\varphi(f^n(x))-\varphi(f^n(y))| \\ 
	& + 
	\frac{1}{N}\sum_{n=n_0}^{N} |\varphi(f^n(x))-\varphi(f^n(y))|\\ 
	& < \frac{\varepsilon}{2}+\frac{\varepsilon}{2} = \varepsilon.
	\end{align*}
	Hence $\underset{N \to +\infty}\lim \Big|S_N (x,\varphi)- S_N (y,\varphi)\Big| = 0$. Since 
	$S_N(x,\varphi) = o(1)$, so \\ $S_N(y,\varphi) = o(1)$.
	This completes the proof.
\end{proof}

\section{\bf The case of graph maps}
The aim of this section is to prove the following theorem:

\begin{thm}\label{t31} Let $G$ be a graph and $f: ~G \to G$ be a continuous map with zero topological entropy.  Then $(1.1)$ holds.
\end{thm}
\medskip

\noindent{}Let us recall the following:

\begin{prop}[\cite{HM}]\label{p33}
	Any $\omega$-limit set of a graph map is either finite set, or an infinite closed nowhere dense set or a 
	finite union of non-degenerate subgraphs (which form a cycle of graphs).
\end{prop}

\begin{defn}[\cite{RS}]
	Let  $f:~G \to G$ be a graph map. A subgraph $K$ of $G$ is called \emph{periodic} of period $k \geq 1$, if
	$K, f(K),\dots,f^{k-1}(K)$ are pairwise disjoint and $f^k (K)=K$. The set Orb$(K)=\cup_{i=0}^{k-1}f^i (K)$ is called a
	\emph{cycle of graphs}. 
\end{defn}

For an infinite $\omega$-limit set $\omega_f(x)$, we let
$$\mathcal{C}(x): = \Big\{X: X\subset G \text{ is a cycle of graphs and } \omega_f(x) \subset X \Big\}.$$ 
The set $\mathcal{C}(x)$ is non-empty by \Big((\cite{RS}, Lemma 9, i), since $f(G) \subset G$\Big).

\begin{defn} An infinite $\omega$-limit set $\omega_f(x)$ is called a \emph{solenoid} whenever 
the periods of the cycles in $\mathcal{C}(x)$ are unbounded.
\end{defn}

Notice that if $\omega_f(x)$ is solenoid, then it is nowhere dense (by Proposition \ref{p33}).
\medskip

\textbf{Case 1}: $\omega_f(x)$ is a solenoid.

\begin{lem} \label{dich}$($\cite{RS}, Lemma 10$)$
	Let $f:~G \to G$ be a graph map and let $\omega_f(x)$ be an infinite $\omega$-limit set. If $\omega_f(x)$ is a solenoid, then there exists a sequence
	of cycles of graphs $(X_n)_{n \geq 1}$ with  increasing periods $(k_n)_{n\geq 1}$ such that, for
	all $n \geq 1$, $X_{n+1} \subset X_n$ and $\omega_f (x) \subset
	\underset{n \geq 1}\bigcap X_n$. Moreover, for all $n \geq 1$, $k_{n+1}$ is
	a multiple of $k_n$ and every connected component of $X_n$ contains the same number (equal to $\frac{k_{n+1}}{k_n}\geq 2$) 
	of components of $X_{n+1}$. Furthermore, $\omega_f (x)$ contains no
	periodic point.
\end{lem}

\begin{prop} \label{C1}
	If $\omega_f(x)$ is a solenoid, then $($\ref{(1.1)}$)$ holds.
\end{prop}

\begin{proof} Fix $x\in X$. Let $\varphi:~G \to \mathbb{R}$ be a continuous function. For any $\varepsilon >0$, there is a function $\phi:~G \to
	\mathbb{R}$ such that $||\varphi-\phi||_{\infty}:= \sup_{x\in G} |\varphi(x) - \phi (x)|<\varepsilon$, where $\phi$ is of 
	the form
	$\phi = \sum_{i=1}^r \alpha_i \psi_{U_i}$, with $\alpha_i \in \mathbb{R}$, $U_i$ is an open free arc in $G$ and 
	$\psi_{U_i}$ is defined as
	follows: $$
	\psi_{U_i} (x) = \left\{
	\begin{array}{lll}
	1 & \mbox{ if }  x\in U_i \\
	\frac{1}{ord(x)} &  \mbox{ if} ~~ x\in \overline{U_i} \backslash U_i\\
	0 & \mbox{ if }  x\in G \backslash \overline{U_i}.
	\end{array}
	\right.
	$$
	By Lemma \ref{dich}, there is a cycle of graphs $X_r$ with period $k_r>0$ such that $\omega_f(x) \subset X_r$. Write
	$X_r=\cup_{i=0}^{k_r -1} f^i (K)$, where $K$ is a subgraph of $G$. There is $s\in \mathbb{N}$ such that $f^s (x) \in K$. 
	Then for any $0\leq i<k_r$ and $n\geq s$, $f^n (x)\in f^i (K)$ if and only if $n\equiv ~ s+i ~ \textrm{mod}(k_r)$. Hence, for any $N>s$,
	
	\begin{align*}
	S_N (x,\psi_{U_j})& =\frac{1}{N}\sum_{n=1}^N \bmu(n) \psi_{U_j}(f^n (x))\\
	& = \frac{1}{N} \sum_{n=1}^{s-1} \bmu(n) \psi_{U_j}(f^n (x))+\frac{1}{N}\sum_{i=0}^{k_r -1}\sum_{s\leq n\leq N, 
		f^n (x)\in f^i (K)}\bmu(n)\psi_{U_j}(f^n (x)).
	\end{align*}
	We distinguish two cases.
	\begin{enumerate}
		
		\item If $f^i (K)\subset U_j$, then by (\ref{l11})
		\begin{align*}
		\frac{1}{N}\sum_{s\leq n\leq N, f^n (x)\in f^i (K)}\bmu(n)\psi_{U_j}(f^n (x))& = \frac{1}{N}\sum_{s\leq n\leq N, 
			n\equiv s+i~ \textrm{mod}(k_r)} \bmu(n)\\
		&=o(1).
		\end{align*}
		\item If $f^i (K)\nsubseteq U_j$ and $f^i (K) \cap U_j \neq \emptyset$, then for any $n,m \geq s$, $n \neq m$, 
		if $f^n (x), f^m(x) \in f^i (K)$, then $|n-m| \geq k_r$. Then by Lemma \ref{l13}, 
		$$\limsup_{N \to +\infty}\frac{1}{N}\Big |\sum_{s\leq n\leq N, f^n(x)\in f^i (K)}\bmu(n)
		\psi_{U_j}(f^n (x))\Big|\leq \frac{1}{k_r}.$$
	\end{enumerate}
	
	The case 2 above can be occurred at most $2$ times and therefore $$\limsup_{N \to +\infty} |S_N(x,\psi_{U_j})| \leq \frac{2}{k_r} .$$
	As $k_r$ is arbitrarily large and  $S_N(x,\phi)=\sum_{i=1}^r \alpha_i S_N (x, \psi_{U_j})$, so $S_N(x,\phi)=o(1)$. 
	Since $||\varphi -\phi ||_{\infty}$ can be taken arbitrarily small, so $S_N(x,\varphi)=o(1)$ and $(\ref{(1.1)})$ holds.
\end{proof}
\bigskip

\textbf{Case 2}: $\omega_f(x)$ is not a solenoid.
\bigskip

Let $X$ be a finite union of subgraphs of $G$
such that $f(X) \subset X$. We define
$$E(X,f)=\Big \{y \in X: \forall \text{ neighborhood } U \text{ of } y \text{ in  } X, \ 
\overline{\mathrm{Orb}_f (U)}=X \Big\}.$$
We call $E(X, f)$ a \emph{basic} set if it is infinite and if $X$ contains a periodic point (cf. \cite{RS}).

\begin{lem}[\cite{RS}, Lemma 12]\label{Min}
	Let $f:~G \to G$ be a graph map and let $\omega_f(x)$ be not a solenoid. There exists a cycle of graphs
	$X \in \mathcal{C}(x)$ such that $ \forall $ $Y \in C(x), \ X \subset Y$. The period of $X$ is maximal among the
	periods of all cycles in $\mathcal{C}(x)$.
\end{lem}

\begin{lem} [\cite{RS}, Lemma 13] \label{inf}
	Let $f :~ G \to G$ be a graph map and let $\omega_f(x)$ be not a solenoid. Let $K$
	be the minimal cycle of graphs in $\mathcal{C}(x)$. Then 
	
	\begin{enumerate}
		\item  For every $y \in \omega_f(x)$ and for every relative neighborhood $U$ of $y$ in $K$,
		$\overline{\mathrm{Orb}_f(U)} = K$.
		\item $ \omega_f(x) \subset E(K, f)$. In particular, $E(K, f)$ is infinite.
	\end{enumerate}
\end{lem}

\begin{lem}[\cite{RS}, Corollary 20]\label{basic}
	If a graph map $f:~G \to G$ admits a basic $\omega$-limit set, then $h(f) > 0$.
\end{lem}

\begin{prop}[\cite{Ma3}, Theorem 5.7] \label{p32} Let $f: G\to G$ be a graph map without periodic points. 
	Then ($G, f$) is a null system.
\end{prop}
It is turns out that the notion of null system is related to the so-called tame system. This later notion was coined by E. Glasner in \cite{Glasner}. The dynamical system $(X,T)$ is tame if the closure of $\big\{T^n / n \in \mathbb{Z}\big\}$ in $X^X$  is Rosenthal compact \footnote{$X^X$ is equipped with the pointwise convergence. This closure is called the enveloping semigroup of Ellis.}. We recall that the set $K$ is Rosenthal compact if and only if there is a Polish space $P$ such that $K \subset {\textrm {Baire-1}}(P)$ where  ${\textrm {Baire-1}}(P)$ is the first class of Baire functions, that is, pointwise limit of continuous functions on $P$. By Bourgain-Fremelin-Talagrand's theorem \cite{BFT}, $K$ is  Rosenthal compact if and only if K is a subset of the Borel functions on $P$ with $K=\overline{\{f_n\}},$ $f_n \in C(P)$.\\
The precise connection between null systems and tame systems is stated in the following proposition:  
\begin{prop}[\cite{G}, \cite{Ke}, \cite{HLSY}] \label{p33} Let ($X, f$) be a dynamical system. If it is a null system then it is tame.
\end{prop}
It is well known that if $(X,T)$ is tame then the pointwise limit of $T$ along any subsequence is Borel, when it exists. 
Combining this with Kushnireko's characterization of the measurable discrete spectrum 
\cite{Kush}\footnote{A transformation measure-preserving has a measurable discrete spectrum if and only if the orbit of 
any square integrable function is compact in $L^2(\mu)$, $\mu$ is an invariant measure.}, it can be seen that tame system 
has a measurable discrete spectrum for any invariant measure. This was observed by W. Huang in \cite{Huang}. 
From this, we see the following: 

\begin{prop}[\cite{HWY1}, Theorem 1.8] \label{p34} Let ($X, f$) be a tame system. Then  $(1.1)$ holds.
\end{prop}

\begin{prop}\label{p310}
	Let $f: G\to G$ be a graph map without periodic points. Then $(1.1)$ holds.
\end{prop}

\begin{proof}By Proposition \ref{p32}, ($G, f$) is a null system and by Proposition \ref{p33},
	($G, f$) is tame. It follows from Proposition \ref{p34} that $(1.1)$ holds.  
\end{proof}
\medskip

\noindent\emph{\textbf{Proof of Theorem \ref{t31}.}}
Let $f:~G \to G$ be a graph map with $h(f)=0$ and let $x\in G$. If $\omega_f(x)$ is finite then $x$ is asymptotic to 
some periodic point. Then by (\ref{l11}) and Lemma \ref{l12}, $($\ref{(1.1)}$)$ holds. 
Now, suppose that $\omega_f(x)$ is infinite. If $\omega_f(x)$ is a solenoid then by Proposition \ref{C1}, 
$($\ref{(1.1)}$)$ holds. Suppose that $\omega_f(x)$ is not a solenoid. Let $X = \cup_{i=0}^{k-1} f^i(K)$ be the minimal 
cycle of $G$ containing $\omega_f(x)$. By Lemma \ref{inf} $($2$)$, $E(X,f)$ is infinite. Then by Lemma \ref{basic}, 
$f$ does not admit a basic set, that is $X \cap P(f) = \emptyset$. For any $0\leq i <k$, set $K_i = f^i(K)$ and $g = f^k$. 
Then $g_i:= g_{|K_i}:~K_i \to K_i$ is a graph map without periodic point. By Proposition \ref{p310}, $(K_i,g_i)$ is a null system and 
therefore so is $(X,f_{|X})$. 
Hence by Proposition \ref{p310}, $(1.1)$ holds for $(X,f_{|X})$. Let $s\geq 0$ such that $f^s (x) \in X$. 
Since $f(X)=X$, there is $y\in X$ such that $f^s (y) =f^s(x)$. In particular, $(x,y)$ is asymptotic. 
Since $S_N(y,\varphi)=o(1)$, so by Lemma \ref{l12}, $S_N(x,\varphi)=o(1)$. This finishes the proof of Theorem \ref{t31}. 
$\blacksquare$
\medskip

\section{\bf The case of local dendrite map}

\subsection{\bf On monotone local dendrite map} 

The aim of this subsection is to prove the following theorem:

\begin{thm} \label{t41} Let $f: ~X \to X$ be a monotone local dendrite map. Then \ $(1.1)$ holds.
\end{thm}

\begin{cor} \label{c41} If $f: ~X \to X$ is a homeomorphism on a local dendrite $X$, then \ $(1.1)$ holds. 
\end{cor}
\medskip

We recall the following results.

\begin{lem}[\cite{am}, Theorem 4.1]\label{l42} Let $f: ~X \to X$ be a monotone local dendrite map. Then $f$ has no Li-Yorke pair. 
	In particular, $f$ has zero topological entropy.
\end{lem}

\begin{lem}[\cite{Ah}, Theorem 1.2]\label{l43}
	Any $\omega$-limit set of a monotone local dendrite map is a minimal set which is either finite, or a Cantor set, or a circle.
\end{lem}

\begin{lem}[\cite{Nagh1}, Corollary 3.7] \label{Cantor}
	Let $f: ~X \to X$ be a monotone dendrite map and $L$ be an infinite $\omega$-limit set.
	Then there is a sequence $\alpha$ of prime numbers such that $f_{|L}$ is topologically conjugate 
	to the adding machine $f_{\alpha}$.
\end{lem}
\medskip
At this point, let us point out that the adding machine satisfy $(1.1)$ since it has a topological discrete spectrum, that is, the eigenfunctions span a dense linear subspace of the $C(X)$ (the space of continuous functions equipped with the strong topology). This will allows us to prove the following:  

\begin{lem}\label{l41} Let $f: ~X \to X$ be a monotone dendrite map. Then \ $(1.1)$ holds.
\end{lem}

\begin{proof} Let $x\in X$ and set $L= \omega_f(x)$.
	If $L$ is finite then $x$ is asymptotic to some periodic point. Then by (\ref{l11}) and Lemma \ref{l12},
	$(1.1)$ holds for the point $x$. Suppose that $L$ is infinite, then it is a Cantor set and 
	then $f_{| L}$ act as the adding machine (Lemma \ref{Cantor}). Hence  $(1.1)$ holds for any point of $L$. But,
	by Lemma \ref{AU}, there exists $y\in L$ such that $(x,y)$ is a proximal 
	pair and by Lemma \ref{l42}, $(x,y)$ is asymptotic. As $(1.1)$ holds for the point $y$, 
	it follows that $(1.1)$ holds for $x$ by Lemma \ref{l12}. 
	This finishes the proof of Lemma \ref{l41}.
\end{proof}
\medskip

Let $X$ be a local dendrite. We define the graph $G_{X}$ as the intersection of all graphs in $X$  containing all the circles. Then $G_{X}$ is a subgraph of $X$ (with $G_{X}=\emptyset$, if $X$ contains no circle).

\begin{prop}[\cite{Ah}, Proposition 3.6]\label{gra}
	Let $f: X\to X$ be a monotone onto local dendrite map.
	Then we have the following properties:\\
	(i) $f(G_{X}) = G_{X}$.\\
	(ii) $f_{/ G_{X}}$ is monotone.
\end{prop}

\begin{lem}\label{l45} Let $f: ~X \to X$ be a monotone onto local dendrite map. Then \ $(1.1)$ holds.
\end{lem}

\begin{proof} 
	Assume that $X$ is not a dendrite and let $x\in X$. Set $L= \omega_f(x)$. By Lemma \ref{AU}, there exists $y\in L$ such that $(x,y)$ is a proximal pair.
	By Lemma \ref{l42}, $(x,y)$ is asymptotic. From Lemma \ref{l43}, we distinguish the following cases.
	
	Case 1: $L$ is finite. In this case, $(1.1)$ holds for the point $x$ similarly as in the proof of Theorem \ref{t41}. 
	
	Case 2: $L$ is a circle: In this case $f_{| L}$ is a circle map, so by Theorem \ref{t31},
	$(1.1)$ holds for the point $y$.
	
	Case 3: $L$ is a Cantor set. In this case, $X$ contains only one circle (i.e. $G_{X} = C$ a circle). 
	If $L$ meets $C$, then $L$ is included in $C$ (by minimality of $L$). Hence by Theorem \ref{t31},
	$(1.1)$ holds for the point $y$. Now if $L$ is disjoint from $C$, then it is 
	included in $X\backslash C$, and then 
	$f_{| L}$ act as the adding machine (Lemma \ref{Cantor}). So $(1.1)$ holds for the point $y$. It follows that $(1.1)$ holds for $x$ by Lemma \ref{l12}. 
	This finishes the proof of Lemma \ref{l45}.
\end{proof}
\medskip

We denote by $\Lambda(f)$ the union of all $\omega$-limit sets of $f$. Define the space $X_{\infty}=\underset{n\in \mathbb{N}}\bigcap f^{n}(X)$. It is a sub-continuum of $X$ and hence $X_{\infty}$ is a sub-local dendrite of $X$. 
Moreover, $X_{\infty}$ is strongly $f$-invariant and we have $\Lambda(f)\subset X_{\infty}.$
\medskip

\begin{lem}[\cite{Ah}, Lemma 4.3]\label{l48} The map $f_{/_{X_{\infty}}}$ is monotone and onto.
\end{lem}
\smallskip

\noindent\textit{\textbf{Proof of Theorem \ref{t41}}}. First by Lemma \ref{l42}, $f$ has zero topological entropy. 
	Let $x\in X$ and set $L= \omega_f(x)$. If $x\in X_{\infty}$, then $(1.1)$ holds for $x$ by Lemmas 
	\ref{l48} and \ref{l45}. 
	Assume that $x\notin X_{\infty}$. By Lemma \ref{AU}, there exists $y\in L$ such that $(x,y)$
	is a proximal pair and by Lemma \ref{l42}, $(x,y)$ is asymptotic. As $L\subset \Lambda(f)\subset X_{\infty}$, 
	so by Lemma \ref{l45}, $(1.1)$ holds for the point $y$ and hence for $x$. The proof is complete. \ $\blacksquare$
\medskip

\subsection{\bf On continuous maps on a certain class of dendrites}

The aim of this subsection is to prove the following Theorem:

\begin{thm}\label{t42}
	Let $X$ be a dendrite such that $E(X)$ is closed and its set of accumulation points $E(X)^{\prime}$ is finite. Let $f:~X \to X$ be a continuous map with zero topological entropy.  Then $(1.1)$ holds.
\end{thm}

We need the following results.

\begin{lem}[\cite{HWZ},\cite{Wei}, Theorem 5.16] \label{cou}
	If $X$ is at most countable and $f:~X \to X$ is a continuous map, then $($1.1$)$ holds.
\end{lem}

\begin{lem}[\cite{Ask} \label{Dec}]
	Let $X$ be a dendrite such that $E(X)$ is closed and $E(X)^{\prime}$ is finite. Let $f:~X \to X$ be a continuous map 
	with zero topological entropy. Let $L$ be an uncountable $\omega$-limit set. Then there is a sequence of
	subdendrites $(D_k)_{k\geq 1}$ of $X$ and a sequence of integers $n_k \geq 2$ 
	for every $k\geq 1$ with the followings properties. For all $k\geq 1$,
	\begin{enumerate}
		\item $f^{\alpha_k}(D_k)=D_k$, where $\alpha_k=n_1 n_2 \dots n_k$,
		\item $\cup_{k=0}^{n_j -1}f^{k \alpha_{j-1}}(D_{j}) \subset D_{j-1}$ for all $j\geq 2$,
		\item $L \subset \cup_{i=0}^{\alpha_k -1}f^{i}(D_k)$,
		\item $f(L \cap f^{i}(D_k)) = L\cap f^{i+1}(D_k)$ for any $0\leq i \leq \alpha_{k}-1$. 
		In particular, $L \cap f^{i}(D_k) \neq \emptyset$,
		\item  $f^{i}(D_k)\cap f^{j}(D_k)$ has empty interior for any $ 0\leq i\neq j<\alpha_k$.
	\end{enumerate}
\end{lem}
\medskip

\emph{Proof of Theorem \ref{t42}}. We distinguish three cases.

\begin{description}
	\item[\textbf{Case 1}]: $L$ is finite. In this case, there is a periodic point $b$ such that 
	$(x,b)$ is asymptotic. Then by (\ref{l11}) and Lemma \ref{l12},  $S_N (x,\varphi)=o(1)$.\\
	
	\item[\textbf{Case 2}]: $L$ is countable. In this case, $Y:=\overline{O_f(x)}=O_f(x) \cup L$ is countable and $f$-invariant. So by Lemma \ref{cou} applied to $(Y,f_{|Y})$, ($1.1$) holds.\\
	
	\item[\textbf{Case 3}]: $L$ is uncountable. Let $\varepsilon >0$ and $k\geq 1$. 
	Let $\varphi \in \mathcal{C}(X, \mathbb{R})$. There exists a 
	function $\varphi_0 \in \mathcal{C}(X, \mathbb{R})$ 
	of the form $\varphi_0=\sum_{k=1}^n \alpha_k \psi_{U_k}$, where $\alpha_k $ are real numbers such that
	$\sup_{x\in X}|\varphi(x)-\varphi_0 (x)| <\frac{\varepsilon}{2}$ and $U_k$ is an open connected subset defined as follows:
	
	$-$  If $U_k \cap E(X)^{\prime}=\emptyset$, then $U_k$ is an open free arc in $ X$.
	
	$-$ If $ U_k \cap E(X)^{\prime} \neq \emptyset$, then $ U_k \cap E(X)^{\prime} =: \{e\}$ and $U_k$ is a connected component of
	$ X \backslash \{z\}$ containing e for some $ z \in X \backslash E(X)$. $~~~~~~~$
	
	$-$ For any $k \neq l$, $U_k \cap U_l = \emptyset$.
	
	The map $\psi_{U_j}$ is defined as follows:
	
	$$
	\psi_{U_j} (x) = \left\{
	\begin{array}{lll}
	1 & \mbox{ if }  x\in U_j \\
	\frac{1}{ord(x)} & \mbox{ if } x\in \overline{U_j} \backslash U_j\\
	0 & \mbox{ if }  x\in X \backslash \overline{U_j}.
	\end{array}
	\right.$$
	Let $X = \cup_{i=0}^{\alpha_k -1} f^i (D_k)$ be as in Lemma \ref{Dec}. There is $n_0\geq 0$ such that 
	$f^{n_0}(x)\in D_k$. Since $D: = \cup_{0\leq i\neq j< \alpha_k}f^i (D_k) \cap f^j (D_k)$ is finite, we may assume that 
	$f^n (x) \notin D$ for any $n\geq n_0$. So for any $n\geq n_0$ and $0\leq s<\alpha_k$, $f^n (x)\in f^s (D_k)$ 
	if and only if $n\equiv n_0 +s ~ \textrm{mod}(\alpha_k)$. Then
	\begin{align*}
	S_N (x,\psi_{U_j})&=\frac{1}{N}\sum_{n=1}^N \bmu(n)\psi_{U_j}(f^n(x))\\
	& = \frac{1}{N}\sum_{n=1}^{n_0 -1} \bmu(n) \psi_{U_j}(f^n(x))+\frac{1}{N} \sum_{n=n_0}^N  \bmu(n) \psi_{U_j}(f^n(x))\\
	& = o(1)+ \sum_{s=0} ^{\alpha_k -1} A_N^s
	\end{align*}
	where $A_n^s = \frac{1}{N} \sum_{n_0 \leq n \leq N, f^n(x) \in f^{s}(D_k)}  \bmu(n) \psi_{U_j}(f^n(x))$.\\ 
	
	For $s=0,1,\dots,\alpha_k -1$, define the sequence $$
	x_n^s = \left\{
	\begin{array}{ll}
	0 & \mbox{ if } ~~ n<n_0 \\
	\psi_{U_j}(f^n(x)) \chi_{f^s (D_k)}(f^n (x)) & \mbox{ if } ~~ n\geq n_0
	\end{array}
	\right.
	$$ where $\chi_{f^s (D_k)}$ is the characteristic function of $f^s (D_k)$. We can rewrite $A_N^s$ as follows:
	\ $A_N^s = \frac{1}{N} \sum_{n=1}^N x_n^s .$ We see that if $f^{s}(D_k) \subset U_j$, then the sequence $(x_n^s)$ is eventually periodic with period $\alpha_k$.
	Then by Lemma \ref{l11}, $A_n^s=\frac{1}{N} \sum_{n=1}^N x_n^s =o(1)$. Indeed, otherwise, there is 
	at most two distinct numbers $s, r \in \{0,1,2,\dots,\alpha_k -1\}$ such that 
	$f^{s}(D_k) \subsetneq U_j$, $f^{s}(D_k) \cap U_j \neq \emptyset$, $f^{r}(D_k) \subsetneq U_j$ and 
	$f^{r}(D_k) \cap U_j \neq \emptyset$. In such case, if $f^n (x), f^p (x) \in f^s (D_k)$ and 
	$n \neq p$, then $|n-p|\geq \alpha_k$. Then by Lemma \ref{l13} $$\limsup_{N \to +\infty} |A_N^s|=\limsup_{N \to +\infty} |A_N^r| \leq \frac{1}{\alpha_k}.$$
	The integer $\alpha_k$ can be taken arbitrary large, then we obtain that \\ $S_N (x,\psi_{U_j}) = o(1)$ and hence 
	$S_N (x,\varphi_0)=o(1)$. Therefore $(1.1)$  holds for $(X,f)$.   \hfill $\blacksquare$
\end{description}
\medskip

\subsection{\bf On a transitive dendrite map with zero entropy}

In \cite{BFK}, Byszewski et al. give an example of transitive map $f$ with zero entropy
on the universal dendrite $D$ with the following properties:
(1) $f$ has a unique fixed point $o$.
(2) $f$ is uniquely ergodic, with the only $f$-invariant Borel probability measure being
the Dirac measure $\delta_o$ concentrated on $o$. Applying the machinery from \cite[p.313]{ALR}, one can see that we have the following. We include the proof for the reader convenience.  

\begin{prop} Let $f$ be the dendrite map above. Then $(1.1)$ holds.
\end{prop}

\begin{proof} Let $\varphi: \ D\longrightarrow \mathbb{R}$ be a continuous function
	and set $\Phi = \varphi-\varphi(o)$.
	As $\delta_o$ is the only $f$-invariant Borel probability measure (by (2)), and 
	since $\int_D|\Phi|d\delta_o = 0$, then
	$\frac{1}{N}\sum\limits_{n=0}^{N-1}|\Phi|(f^{n}(x))\longrightarrow 0$.
	As 
	$$\Big|\frac{1}{N}\sum_{n=0}^{N-1}\bmu(n)\varphi(f^n(x))\Big|
	\leq \frac{1}{N}\sum\limits_{n=0}^{N-1}|\Phi|(f^{n}(x)) +
	|\varphi(o)| \Big|\frac{1}{N}\sum_{n=0}^{N-1}\bmu(n)\Big|$$ and as
	$\frac{1}{N}\sum_{n=0}^{N-1}\bmu(n)\longrightarrow 0$ (by Prime Number Theorem), so we get that
	$\frac{1}{N}\sum_{n=0}^{N-1}\bmu(n)\varphi(f^n(x))\longrightarrow 0$.
\end{proof}
\bigskip

Let us notice that the convergence in $(1.1)$ is uniform, since $(D,f)$ satisfy the so-called MOMO property (M\"{o}bius Orthogonality on Moving Orbits) (see \cite{AKLR2} for the definition). In this direction, it is proved in \cite{AKLR2} the following
\begin{prop}[\cite{AKLR2}]\label{Uniform}
	If Sarnak's conjecture $(1.1)$ is true then for all zero entropy systems
	$(X, T )$ and $f\in C(X)$, then we have 
	$$\frac1N \sum_{n=1}^{N}f(T^nx)\bmu(n) \tend{N}{+\infty}0,$$
	uniformly in $x \in X$.
\end{prop}  
\medskip

\subsection{\bf On a transitive dendrite map with positive entropy}
In this subsection, we discuss the problem of M\"{o}bius disjointness for the example introduced by 
\v{S}pitalsk\'{y} in \cite{Sp}. V. \v{S}pitalsk\'{y} constructed his example as a factor of a map $F$ acting 
on the universal dendrite of order $3$. Precisely, let $Q$ be a set of all dyadic rational numbers in $(0,1)$, 
that is, every $r \in Q$ can be uniquely written as $r=\frac{p_r}{2^{q_r}}$ with $q_r \geq 1$ and $p_r$ is 
odd in $ \{1,\cdots,2^{q_r}\}$. Let us denote by $Q^0$, $Q^*$ the sets $\{0\}$ and $\cup_{k \geq 0}Q^k$. 
The length of element $\alpha \in Q^*$ denote by $|\alpha|$ correspond to the integer $k$ such that 
$\alpha \in Q^k$. We define on $Q^*$ the concatenation operation as follows:\\
$\textrm{For} \ \alpha \in Q^k, \beta \in Q^m$, we put $\gamma=\alpha \beta \in Q^{k+m}$. 
If $\alpha=r_0r_1\cdots r_{k-1} \in Q^k$ with $k \geq 1$, then $\widetilde{\alpha}$ denotes 
$r_0r_1\cdots r_{k-2}.$
The dendrite of order $3$ is given by     
$$X=\overline{\bigcup_{m \geq 0}X^{(m)}}=
\bigcup_{m \geq 0}X^{(m)} \cup X^{(\infty)},$$
where $X^{(0)}=[a_0,b_0]$ is an arc and every $X^{(m)}$, $m \geq 1$, satisfies
$$X^{(m)}=X^{(m-1)} \cup \Big(\bigcup_{\alpha \in Q^m}(a_\alpha,b_\alpha]\Big),$$
$a_\alpha \in (a_{\widetilde{\alpha}},
b_{\widetilde{\alpha}}],$ for $\alpha \in Q^m;$ moreover
$$B(X)=\bigcup_{k \geq 1}Q^k \textrm{~~and~~}
E(X)=\{a_0,b_0\} \cup \{b_\alpha, \alpha \in B(X)\} \cup X^{(\infty)}.$$ 
For any $\alpha \in Q^*$, we denote by $A_\alpha=[a_\alpha,b_\alpha]$ and by $X_\alpha$ the closure of 
the component of $X \setminus \{a_\alpha\}$ containing $b_\alpha$. Therefore,
\[X_\alpha=
\begin{cases}
X_o   & \text{if}~~~~\alpha=o, \\
\Big(\bigcup_{\beta \in Q^*} A_{\alpha\beta}\Big) \bigcup
\Big\{b_{\alpha\beta}:~~\beta \in Q^{\infty}\Big\} & \text{if}~~~~\alpha \in Q^*,
\end{cases}
\]
where, for any $\nu=\nu_0\cdots \in Q^{\infty}$, $b_\nu$ denotes the unique
point of $\bigcap_{m \geq 1}X_{\nu_0\cdots\nu_{m-1}}$. Notice that we have
$$X^{(\infty)}=\Big\{b_{\alpha}:~~~\alpha \in Q^{\infty}\Big\}.$$
By construction of the map $F$, for every $\alpha \in Q^*$ there is unique $\rho(\alpha) \in Q^*$ with 
$F(b_\alpha)=b_{\rho(\alpha)}$. This definition can be extended to $\alpha \in Q^{\infty}$ 
(see Lemma 8 in \cite{Sp}). This allows us to define $F$ on $X^{(\infty)}$ by 
$$F(b_\alpha) = b_{\rho(\alpha)}, ~~~~~~\text{for~every~} \alpha \in Q^{\infty}.$$

 V. \v{S}pitalsk\'{y} proved that $F$ has positive entropy and for any $x\in X^{(m)}$, $m \geq 1$, the omega set of $x$ is either 
$\{a_0\}$ or $\{b_0\}.$ Furthermore, if $\nu$ is an $F$-invariant Borel probability measure, then for each $m\geq 1$, 
$\nu (X^{(m)}\backslash \{a_0, b_0\})= 0$, and  $a_0$ and $b_0$ are the only fixed points of $F$. This yields that the 
topological entropy of  $F_{|_{X^{(m)}}}$ is zero.  Therefore, by (\cite{Bowen}, Proposition 2, (c)), the entropy of 
$F_{|\cup_{m \geq 1}X^{(m)}}$ is zero.
Moreover, by the same arguments as before, we can see easily that 
for any $x\in \cup_{m \geq 1} X^{(m)}$, for any continuous function $\Phi$, we have \\
$$\frac{1}{N}\sum_{n=1}^{N}\Phi(F^n(x))\bmu(n) \tend{N}{+\infty}0.$$
 
According to Proposition \ref{Uniform}, if Sarnak conjecture is true then the M\"{o}bius disjointness is uniform. 
But, we can not apply this result in our situation since the set $\cup_{m \geq 1}  X^{(m)}$ is a $F_\sigma$. 
Although, the M\"{o}bius disjointness holds uniformly on each $X^{(m)}$. We thus asked whether 
\v{S}pitalsk\'{y}'s example satisfy M\"{o}bius disjointness or not. This allows us also to ask the following questions.

\begin{Que} Let $(X,F)$ be the \v{S}pitalsk\'{y}'s example. Do we have that 
the M\"{o}bius disjointness is true for $(X,F)$?
\end{Que}
\medskip 

\begin{Que} Let $(X,T)$ be a dynamical system with zero topological entropy. Let $Y$ be a dense $T$-invariant subset of $X$. 
Again by Proposition 2. (c) from \cite{Bowen}, the topological entropy $T|_Y$ is zero. 
Assume that the M\"{o}bius disjointness for $(Y,T|_Y)$ holds, do we have that Sarnak conjecture is true for $(X,T)$?
\end{Que}

\textit{Acknowledgements.} The research of H. Marzougui and G. Askri acknowledge support from the research unit:
``Dynamical systems and their applications'' (UR17ES21), of Higher Education and Scientific Research, Tunisia.  The authors thanks Issam Naghmouchi for fruitful discussions on this paper.

\bibliographystyle{amsplain}

\end{document}